\documentclass{amsart}

\usepackage{graphicx,amssymb,amsmath,xcolor,enumerate}
\usepackage{hyperref}
\hypersetup{
	bookmarks=true,         
	pdffitwindow=false,     
	pdfstartview={FitH},    
	colorlinks=true,      
	citecolor=red,
}

\usepackage{float}

\makeatletter
\def\l@subsection{\@tocline{2}{0pt}{2.5pc}{5pc}{}}
\makeatother

\usepackage{geometry}
\DeclareMathOperator{\sgn}{sgn}

\DeclareSymbolFont{largesymbol}{OMX}{yhex}{m}{n}
\DeclareMathAccent{\Widehat}{\mathord}{largesymbol}{"62}
\newcommand*\di{\mathop{}\!\mathrm{d}}

\def\e{\varepsilon}

\newcommand{\tcr}[1]{\textcolor{red}{#1}}

\numberwithin{equation}{section}              
\newtheorem{theorem}{Theorem}[section]

\newtheorem{lemma}{Lemma}[section]

\newtheorem*{proposition*}{Proposition}

\newtheorem*{corollary*}{Corollary}
\newtheorem{definition}{Definition}[section]
\newtheorem*{definitions*}{Definitions}

\newtheorem*{acknowledgements*}{Acknowledgements}

\newtheorem*{conjecture*}{\bf Conjecture}

\newtheorem*{example*}{\bf Example}
\theoremstyle{remark}
\newtheorem{remark}{\bf Remark}[section]

\begin{document}
\date{}                                     

\author{Yu Gao}
\address[Y. Gao]{Department of Applied Mathematics, The Hong Kong Polytechnic University, Hung Hom, Kowloon, Hong Kong}
\email{mathyu.gao@polyu.edu.hk}

\author{Hao Liu}
\address[H. Liu]{School of Mathematical Sciences and Institute of Natural Sciences, Shanghai Jiao Tong University, Shanghai, China,}
\email{mathhao.liu@sjtu.edu.cn}

\author{Tak Kwong Wong}
\address[T. K. Wong]{Department of Mathematics, The University of Hong Kong, Pokfulam, Hong Kong.}
\email{takkwong@maths.hku.hk}

\title[Stability of peakons]{Stability of peakons of the Camassa-Holm equation beyond wave breaking}

\maketitle

\begin{abstract}
Using a generalized framework that consists of evolution of the solution to the Camassa-Holm equation and its energy measure, we establish the global-in-time orbital stability of peakons with respect to  the perturbed (energy) conservative solutions to the Camassa-Holm equation. Especially, we extend the $H^1$-stability result obtained by Constantin and Strauss (Comm. Pure Appl. Math., 53(5), 603-610, 2000) globally-in-time, even after the perturbed solutions experience wave breaking. In addition, our result also shows that the singular part of the energy measure of the perturbed solutions will remain stable for all times. 
\end{abstract}
{\small}

{\small}

\section{Introduction}\label{sec:intro}
In this paper, we study the global-in-time orbital stability of peakons of the Camassa-Holm equation on the whole real line $\mathbb{R}$. For  any $x$, $t\in\mathbb{R}$, the Camassa-Holm equation is given by
\begin{align}\label{eq:CH}
u_t+uu_x+P_x=0,
\end{align}
where $u=u(x,t)$ is the unknown function.
Here, the function $P:=P(x,t)$ is a nonlocal source term given by
\[
P(x,t)= \frac{1}{2} \left[\varphi \ast \left(u^2+\frac{u_x^2}{2}\right)\right](x,t)=\frac{1}{2}\int_{\mathbb{R}}\varphi(x-y)\left(u^2+\frac{u_x^2}{2}\right)(y,t)\di y,
\]
where $\frac{1}{2} \varphi(x) :=\frac{1}{2}e^{-|x|}$ is the fundamental solution to the Helmholtz operator $I-\partial_{xx}$. Equation~\eqref{eq:CH} is an integrable system that has a bi-Hamiltonian structure and infinitely many conservation laws \cite{camassa1993integrable,fuchssteiner1981symplectic}. It is used to  model the propagation of unidirectional shallow water waves over a flat bottom \cite{camassa1993integrable,johnson2002camassa}.  Similar to the Korteweg-de Vries (KdV) equation, 
Equation~\eqref{eq:CH} has  a special class of soliton solutions known as peakons, which are in the form 
\[
c\varphi(x-ct)= ce^{-|x-ct|},
\]
where $c\in\mathbb{R}$ is any nonzero constant.  We are going to study the global-in-time orbital stability of these peakons. To illustrate our results clearly, let us first introduce the results of (energy) conservative solutions, as well as previous results of the stability of peakons.

For any initial data $\bar{u}\in H^s(\mathbb{R})$ with $s>3/2$, the Camassa-Holm equation has a unique solution in $u\in C([0,T];H^s(\mathbb{R}))$ for some $T>0$, and for these solutions, the following two quantities are conserved \cite{li2000well,rodriguez2001cauchy}:
\begin{align}\label{eq:constant}
E(u)=\int_{\mathbb{R}}(u^2+u_x^2)\di x \quad \mbox{and}\quad F(u)= \int_{\mathbb{R}}(u^3+uu_x^2)\di x.
\end{align}
It was then proved  in  \cite{constantin2000global} that if the initial datum $\bar{u}\in H^1(\mathbb{R})$ satisfies $\bar{u}-\bar{u}_{xx}\in\mathcal{M}_+(\mathbb{R})$,  the Camassa-Holm equation has a unique global-in-time solution $u\in C([0,\infty);H^1(\mathbb{R}))$ with $E(u)$ and $F(u)$ conserved. Here, $\mathcal{M}_+(\mathbb{R})$ stands for the set of all finite positive Radon measures on $\mathbb{R}$. On the other hand, even for  smooth initial data, the solutions to the Camassa-Holm equation might blow up in finite time  \cite{brandolese2014local,camassa1993integrable,camassa1994new,constantin2000existence,constantin1998wave,mckean2004breakdown}, and at the blow-up time the solution $u$ remains continuous while its derivative $u_x$ tends to negative infinity. This blow-up phenomenon is known as the wave breaking in the literature. Before the perturbed solution $u$ experience wave breaking, the orbital stability of peakons was proved by Constantin and Strauss in \cite{constantin2000stability} based on a very clever use of the two conserved quantities $E(u)$ and $F(u)$ defined by \eqref{eq:constant}. Indeed, it was commented in \cite[Section~3]{constantin2000stability} that their stability result is applicable only up to the wave breaking time of the perturbed solution $u$, as at the blow-up time there might be some ``energy loss" in the sense that the energy $E(u)$ defined by \eqref{eq:constant} is strictly smaller than the initial one, and the perturbed solution $u(\cdot,t)$ is not continuous in $H^1(\mathbb{R})$ at the blow-up time; see \cite{bressan2007global,bressan2005optimal,holden2007global}  for instance. Therefore, the stability results in  \cite{constantin2000stability} cannot be used beyond  the wave breaking time of the perturbed solution $u$.

However, the peakons exists globally, and for any initial data  $\bar{u} \in H^1(\mathbb{R})$, the solutions $u(\cdot,t) $ can be extended globally-in-time and uniquely after the wave breaking time in a conservative way, so that the solution $u(\cdot,t) \in H^1(\mathbb{R})$ for all $t\in \mathbb{R}$ and the energy $E(u)$ is conserved for a.e. $t\in \mathbb{R}$; see \cite{bressan2007global,bressan2005optimal,holden2007global} for instance. Moreover, the  energy $E(u)$ is strictly dissipative at some zero measure set of time $t$.  Hence, the nature question is: do we still have the orbital stability of peakons with respect to these global conservative solutions for all time $t$, especially after the wave breaking time?
As mentioned in \cite{el2009stability,molinet2018liouville} for instance, this still remains to be an open question since these conservative solutions fail to be continuous in $H^1(\mathbb{R})$ at the aforementioned zero measure set of time $t$.
The aim of this paper is to provide an affirmative answer to this question by using a generalized framework of the Camassa-Holm equation.

As mentioned above, the energy $E(u)$ might be strictly dissipative at the blow-up time; however, the conservations of $E(u)$ and $F(u)$ are essential for obtaining the orbital stability of peakons in \cite{constantin2000stability} before the wave breaking. Then it is quite natural to find some substitutes of $E(u)$ and $F(u)$, so that they are conserved for all time and we can utilize them in the same way as $E(u)$ and $F(u)$.  The strict dissipation of the energy $E(u)$ at the blow up time is because some of the total energy is transferred into a singular measure, which cannot be measured/detected by the solution $u$ itself; see \cite{bressan2007global,holden2007global} for instance. To describe the evolution of the total energy, we shall introduce the energy measure $\mu(t)$ which is conserved globally-in-time (i.e., ${\mu}(t)(\mathbb{R}) = {\mu}(0)(\mathbb{R})$ for all $t\in \mathbb{R}$) and $\di\mu_{ac}=(u^2+u_x^2)\di x$, where $\mu_{ac}$ is the absolutely continuous part of $\mu$ with respect to the Lebesgue measure; see \cite[Section 3]{holden2007global}. 
Based on this idea, the generalized framework of the Camassa-Holm equation is described as follows.
Assume that $u$ is sufficiently smooth, such as $u\in C^1_tC^1_x\cap C^0_tC^2_x$. Differentiating Equation~\eqref{eq:CH} with respect to the spatial variable $x$ yields
\begin{align}\label{eq:derivative of CH}
u_{xt}+u_x^2 + uu_{xx}+P_{xx}=0.
\end{align}
Multiplying \eqref{eq:CH} by $2u$ and multiplying \eqref{eq:derivative of CH} by $2u_x$, and then summing them up, we formally obtain the conservation equation for the energy density:
\begin{align}\label{eq:energy}
(u^2+u_x^2)_t+[u(u^2+u_x^2)]_x=[u(u^2-2P)]_x.
\end{align}
To describe the evolution of the total energy we may replace  $u^2+u_x^2$ in \eqref{eq:energy} with $\mu(t)$. Then supplement the original Camassa-Holm equation \eqref{eq:CH} with the resulting equation,
 the generalized framework for the Camassa-Holm equation is given by:
\begin{align}
	&u_t+uu_x+P_x=0,\label{eq:gCH1}\\
	&(\di\mu)_t+(u\di\mu)_x=[u(u^2-2P)\di x]_x,\label{eq:gCH2}\\
	&\di\mu_{ac}(t)=(u^2+u_x^2)(\cdot,t)\di x,\label{eq:gCH3}
\end{align}
where  the function $P$ now becomes $P:=\frac{1}{4}\varphi\ast\mu+\frac{1}{4} \varphi\ast u^2$, and $\varphi (x) := e^{-|x|}$. Here, we understand Equation~\eqref{eq:gCH2} in the distributional sense \eqref{eq:fourth}. 
The conservative solutions to the above generalized framework \eqref{eq:gCH1}-\eqref{eq:gCH3} is considered in the following space:
\begin{align}\label{eq:solution pair}
	\mathcal{D}=\Big\{(u,\mu):~u\in H^1(\mathbb{R}),~~\mu\in\mathcal{M}_+(\mathbb{R}),~~\di \mu_{ac}=(u^2+u_x^2)\di x\Big\}.
\end{align}
When the energy measure $\mu(t)$ is absolutely continuous with respect to the Lebesgue measure $\di x$, the above framework is the same as Equations~\eqref{eq:CH} and \eqref{eq:energy}. The above generalized framework is similar to the generalized framework for the Hunter-Saxton equation; see \cite[Eqt. (1.4)-(1.6)]{gao2021regularity} for instance. Similar ideas were also used in other models, such as the two component Camassa-Holm equaiton \cite{grunert2012global}, the Hunter-Saxton equation \cite{bressan2007asymptotic,antonio2019lipschitz}, and the Novikov equation \cite{chen2015existence}.
The definition of conservative solutions to the generalized framework \eqref{eq:gCH1}-\eqref{eq:gCH3} is as follows: 
\begin{definition}[Conservative solutions]\label{def:Conservative solutions}
For any initial datum $(\bar{u}, \bar{\mu}) \in \mathcal{D}$, the pair $(u(\cdot, t), \mu(t))$
is said to be a global-in-time conservative solution to the generalized framework  \eqref{eq:gCH1}-\eqref{eq:gCH3} subject to the initial datum $(\bar{u}, \bar{\mu})$, if the pair $(u(\cdot,t),\mu(t))$ satisfies
\begin{enumerate}[(i)]
\item $u \in L^\infty(\mathbb{R}; H^1(\mathbb{R}))$ and $u(\cdot,t) \in H^1(\mathbb{R})$ for all $t$ $\in \mathbb{R}$. Moreover, for any $T>0$, the map $t \mapsto u(\cdot, t)$ is Lipschitz continuous from $[-T,T]$ into $L^2(\mathbb{R})$, and $\mu\in C(\mathbb{R};\mathcal{M}_+(\mathbb{R}))$;
\item $(u(\cdot,0),\mu(0))=(\bar{u},\bar{\mu})$, and $\di\mu(t)= (u^2+u_x^2)(\cdot,t)\di x$ for a.e. $t\in\mathbb{R}$; 
\item Equation~\eqref{eq:gCH1} is satisfied in the weak sense: for any test function $\phi\in C_c^\infty(\mathbb{R}\times \mathbb{R})$,
\begin{align}\label{eq:weakformula}
\int_{\mathbb{R}}\int_{\mathbb{R}}u\phi_t-\phi\left(uu_x + P_x\right)\di x\di t=0;
\end{align}
\item (Conservation of energy) Equation~\eqref{eq:gCH2} is satisfied in the distributional sense: for any test function $\phi\in C^{\infty}_c(\mathbb{R}\times \mathbb{R})$, the identity 
\begin{equation}\label{eq:fourth}
\int_{\mathbb{R}}\int_{\mathbb{R}} \left( \phi_{t} + u \phi_{x}\right) \di \mu(t) \di t = \int_{\mathbb{R}}\int_{\mathbb{R}} \phi_{x} u (u^2-2P) \di x \di t
\end{equation}
holds; and
\item Equation~\eqref{eq:gCH3} holds for all $t \in \mathbb{R}$. 
\end{enumerate}
The last condition is equivalent to say that $(u(t),\mu(t))\in \mathcal{D}$  for all times $t\in \mathbb{R}$.
\end{definition}
The existence and uniqueness of conservative solutions ($u, \mu$) in the sense of Definition \ref{def:Conservative solutions} with initial datum $(\bar{u},\bar{\mu})\in \mathcal{D}$ has been obtained based on the characteristics method; see \cite{bressan35uniqueness,bressan2007global,holden2007global} for example.  
Our main result for stability of peakons is as follows:
\begin{theorem}\label{thm:onepeakonthm}
Let $0<\e<1$ and $c>0$ be any given constants, and the initial data $(\bar{u},\bar{\mu})\in\mathcal{D}$, where the space $\mathcal{D}$ was defined in \eqref{eq:solution pair} above, satisfy
\[
\bar{\mu}_s(\mathbb{R})+\|\bar{u}-c\varphi\|_{H^1}^2\leq\frac{c^2}{2}\left[\frac{\e}{6(1+c)}\right]^8,
\]
where $\varphi(x) :=e^{-|x|}$, and $\bar{\mu}_s$ is the singular part of the initial energy measure $\bar{\mu}$.
If $(u,\mu)$ is a conservative solution to the generalized framework \eqref{eq:gCH1}-\eqref{eq:gCH3} of the Camassa-Holm equation subject to the initial datum $(\bar{u},\bar{\mu})$ in the sense of Definition \ref{def:Conservative solutions},
then for any $t\in\mathbb{R}$, we have
\[
\mu_s(t)(\mathbb{R})+\|u(\cdot,t)-c\varphi(\cdot-\xi(t))\|_{H^1}^2\leq \e^2,
\]
where $\mu_s(t)$ is the singular part of the energy measure $\mu(t)$, and the spatial location $\xi(t)\in\mathbb{R}$ is any point where the function $u(\cdot,t)$ attains its maximum.
\end{theorem}
Since the Camassa-Holm equation \eqref{eq:CH} is invariant under the transformation: $u(x,t)\mapsto -u(-x,t)$, one can obtain similar results for the anti-peakon cases, i.e., $c<0$.
Compared with previous results for the stability of peakons, our result shows the following novelties:
\begin{enumerate}[(i)]
\item The orbital stability of peakons in the $H^1(\mathbb{R})$ norm is obtained for all times $t$, so it applies even after the wave breaking time of the perturbed solution $u$.
\item Our results give the estimates for $u$ and $\mu_s(t)$ together; in particular, the singular part $\mu_s(t)$ of the energy measure will also remain stable for all times $t$.
\end{enumerate}

The rest of this paper is organized as follows. In Section \ref{sec:general},  we will introduce the characteristics method for this generalized framework and present an equivalent semilinear system. Using the energy measure $\mu$, we will then define the extensions of $E(u)$ and $F(u)$ in the generalized framework; see \eqref{eq:newenergy} below. In addition, we will also show that these extensions are actually conserved for all times. The main result for stability of peakons will be proved in Section~\ref{sec:stable} by utilizing these two conserved quantities.

\section{Generalized framework and conservative quantities}\label{sec:general}

\subsection{Generalized framework and system of $\alpha$-variable}\label{sec:charac}
 
The existence and uniqueness of conservative solutions ($u, \mu$) in the sense of Definition \ref{def:Conservative solutions} with initial datum $(\bar{u},\bar{\mu})\in \mathcal{D}$ has been obtained in \cite{bressan35uniqueness,bressan2007global,holden2007global}. The idea is to use the characteristics method to transfer the Camassa-Holm equation into a semilinear system. 
Next, we follow the same ideas in \cite{bressan35uniqueness,bressan2007global,holden2007global} and introduce some details for a similar but slightly different self-contained semilinear system of ordinary differential equations, which is equivalent to the system \eqref{eq:gCH1}-\eqref{eq:gCH3}. 
To this end, we formally assume $\di\mu(t)=(u^2+u_x^2)(\cdot,t)\di x$ for all $t\in\mathbb{R}$. Let us introduce the $\alpha$-coordinate as follows: for any $(\bar{u},\bar{\mu})\in\mathcal{D}$, we define the function $\bar{x}$ via
\begin{align}
\bar{x}(\alpha)+\bar{\mu}((-\infty,\bar{x}(\alpha)))\leq \alpha\leq \bar{x}(\alpha)+\bar{\mu}((-\infty,\bar{x}(\alpha)])
\end{align}
for every $\alpha \in \mathbb{R}$.
The use of the $\alpha$-variable transfers the initial energy measure $\bar{\mu}$ into another measure with the density function $1-\bar{x}'$ in $L^1(\mathbb{R})$, namely
\[
\di \bar{\mu}=\bar{x}\# [(1-\bar{x}')\di \alpha].
\]
Here, the symbol $\#$ stands for push-forward of a measure, i.e., for any Lebesgue measurable set $A\subset\mathbb{R}$, we have $\bar{\mu}(A)=\int_{\bar{x}^{-1}(A)} (1-\bar{x}')(\alpha) \di \alpha$. This function $\bar{x}:=\bar{x}(\alpha)$ is Lipschitz continuous with a Lipschitz constant bounded by $1$; see \cite[Proposition 2.1]{gao2021regularity} for instance.  Next, we consider the flow map satisfying
\begin{align}\label{eq:ut}
\frac{\di}{\di t}y(\alpha,t)=u(y(\alpha,t),t),\quad y(\alpha,0)=\bar{x}(\alpha),
\end{align}
which implies
\begin{align}\label{eq:flowmap}
y(\alpha,t)=y(\alpha, 0)+\int_0^tu(y(\alpha,s),s)\di s.
\end{align}
It follows from \eqref{eq:ut} and the energy equation \eqref{eq:energy} that 
\begin{align}
\frac{\di}{\di t}\int_{-\infty}^{y(\alpha,t)}(u^2+u_x^2)(x,t)\di x=(u^3-2uP)(y(\alpha,t),t),
\end{align}
and hence, a direct integration yields
\begin{align}\label{eq:energyflow}
\int_{-\infty}^{y(\alpha,t)}(u^2+u_x^2)(x,t)\di x=\int_{-\infty}^{y(\alpha,0)}(\bar{u}^2+\bar{u}_x^2)(x)\di x+\int_0^t(u^3-2uP)(y(\alpha,s),s)\di s.
\end{align}
Combining \eqref{eq:flowmap} and \eqref{eq:energyflow} yields
\begin{equation}\label{eq:combine}
\begin{aligned}
y(\alpha,t)+\int_{-\infty}^{y(\alpha,t)}(u^2+u_x^2)(x,t)\di x=\alpha+\int_0^t(u+u^3-2uP)(y(\alpha,s),s)\di s.
\end{aligned}
\end{equation}
Define
\begin{align}\label{eq:def of beta}
\beta(\alpha,t):=y(\alpha,t)+\int_{-\infty}^{y(\alpha,t)}(u^2+u_x^2)(x,t)\di x.
\end{align}
Then, we have $\beta(\alpha,0)=\alpha$ 
and
\begin{align}\label{eq:betat}
\frac{\di}{\di t}\beta(\alpha,t)=(u+u^3-2uP)(y(\alpha,t),t).
\end{align} 
we can re-write \eqref{eq:def of beta} as
\begin{align}
\int_{-\infty}^{y(\alpha,t)}(u^2+u_x^2)(x,t)\di x=\beta(\alpha,t)-y(\alpha,t),
\end{align}
and hence, a direct differentiation with respect to $\alpha$ yields
\[
\partial_\alpha[\beta(\alpha,t)-y(\alpha,t)]=(u^2+u_x^2)(y(\alpha,t),t)\partial_\alpha y(\alpha,t).
\]
Therefore, we formally have
\begin{align}\label{eq:mut}
\di \mu(t) = y(\cdot,t)\# \Big(\partial_\alpha[\beta(\alpha,t)-y(\alpha,t)] \di \alpha\Big).
\end{align}
Using the above relations, we also have
\begin{align}\label{eq:P}
P(y(\alpha,t),t)=\frac{1}{4}\int_{\mathbb{R}}e^{-|y(\alpha,t)-y(\theta,t)|}\Big[\beta_\theta(\theta,t)-y_\theta(\theta,t)+u^2(y(\theta,t),t)y_\theta(\theta,t)\Big]\di\theta,
\end{align}
and
\begin{align}\label{eq:Pderivative}
P_x(y(\alpha,t),t)=-\frac{1}{4}\int_{\mathbb{R}}\sgn(\alpha-\theta)e^{-|y(\alpha,t)-y(\theta,t)|}\Big[\beta_\theta(\theta,t)-y_\theta(\theta,t)+u^2(y(\theta,t),t)y_\theta(\theta,t)\Big]\di\theta.
\end{align}
Combining \eqref{eq:CH}, \eqref{eq:ut}, and \eqref{eq:betat} yields the following self-contained system:
\begin{align}
&\frac{\di}{\di t}y(\alpha,t)=u(y(\alpha,t),t),\quad y(\alpha,0)=\bar{x}(\alpha),\label{eq:g1}\\
&\frac{\di}{\di t}\beta(\alpha,t)=(u+u^3-2uP)(y(\alpha,t),t),\quad \beta(\alpha,0)=\alpha,\label{eq:g2}\\
&\frac{\di}{\di t}u(y(\alpha,t),t)=-P_x(y(\alpha,t),t),\quad u(y(\alpha,0),0)=\bar{u}(\bar{x}(\alpha)).\label{eq:g3}
\end{align}
For any initial datum $(\bar{u},\bar{\mu})\in\mathcal{D}$, the global solutions $(y(\alpha,t),\beta(\alpha,t), u(y(\alpha,t),t))$ to this system, namely Equation~\eqref{eq:g1}-\eqref{eq:g3}, was obtained in \cite{holden2007global}. It is worth noting that the notations used in \cite{holden2007global} are slightly different from ours, where they used $H(\alpha,t):=\beta(\alpha,t)-y(\alpha,t)$ as one of the unknown functions instead. Subtracting \eqref{eq:g1} from \eqref{eq:g2} yields the evolution equation for $H(\alpha,t)$.  
\begin{remark}\label{rmk:regularity}
Let $\zeta(\alpha,t):=y(\alpha,t)-\alpha$, $U(\alpha,t):=u(y(\alpha,t),t)$, and $H(\alpha,t):=\beta(\alpha,t)-y(\alpha,t)$. Then it follows from \cite[Theorem 2.8]{holden2007global} that
\[
(\zeta,U,H)\in C^1(\mathbb{R};E\cap [W^{1,\infty}(\mathbb{R})]^3),
\]
where $E=V\times H^1(\mathbb{R})\times V$ and $V=\{f\in C_b(\mathbb{R}):~~f_x\in L^2(\mathbb{R})\}$. Moreover, $y_\alpha(\cdot,t)\geq 0$, $H_\alpha(\cdot,t)=\beta_\alpha(\cdot,t)-y_\alpha(\cdot,t)\geq 0$, and  $P$, $P_x$ defined by \eqref{eq:P} and \eqref{eq:Pderivative} (as functions of $(\alpha,t)$) are in the space $C(\mathbb{R};H^1(\mathbb{R}))$; see \cite[Lemma 2.1]{holden2007global} for details. 
\end{remark}
For any given global solution $(y(\alpha,t),\beta(\alpha,t),u(y(\alpha,t),t))$ to the system \eqref{eq:g1}-\eqref{eq:g3}, the corresponding global conservative solutions $(u(t),\mu(t))$ to the generalized framework \eqref{eq:gCH1}-\eqref{eq:gCH3} will be given by  $u(x,t)=u(y(\alpha,t),t)$ for $x=y(\alpha,t)$ and global energy measure $\mu(t)$ will be given by \eqref{eq:mut}. The uniqueness of conservative solutions was proved via characteristics methods in \cite{bressan35uniqueness}. Since the global solution $(y(\alpha,t),\beta(\alpha,t), u(y(\alpha,t),t))$ to \eqref{eq:g1}-\eqref{eq:g3} can immediately provide the global conservative solution $(u(t),\mu(t))$ to the generalized framework \eqref{eq:gCH1}-\eqref{eq:gCH3}, this is equivalent to say that any (energy) conservative solution in the sense of Definition~\ref{def:Conservative solutions} can be represented by the solution to \eqref{eq:g1}-\eqref{eq:g3}.

\subsection{Conservative quantities under generalized framework}
As mentioned in Section~\ref{sec:intro}, the two quantities $E(u)$ and $F(u)$ defined in \eqref{eq:constant} are not conserved in general, due to the potential wave breaking. Under the generalized framework, for a pair $(u,\mu)\in\mathcal{D}$, the natural generalizations of $E(u)$ and $F(u)$ are defined by 
\begin{align}\label{eq:newenergy}
\tilde{E}(u, \mu)=\mu(\mathbb{R}) \quad\mbox{and}\quad \tilde{F}(u,\mu)=\int_{\mathbb{R}}u\di \mu.
\end{align}
Here, the generalized quantity $\tilde{E}$ indeed depends on the variable $\mu$ only, however, we still use the pair  $(u,\mu)$ in its argument to emphasis that $\tilde{E}$ is the total energy for the generalized solution $(u(t),\mu(t))$.
For a conservative solution $(u,\mu)$ to the Camassa-Holm equation in the sense of Definition~\ref{def:Conservative solutions}, we know that the (generalized) energy $\tilde{E}(u(t), \mu(t))$ is conserved, namely $\tilde{E}(u(t), \mu(t)) =\mu(t)(\mathbb{R}) = \bar{\mu}(\mathbb{R}) = \tilde{E}(\bar{u}, \bar{\mu})$ for all $t\in\mathbb{R}$. This energy conservation can be easily verified by using \eqref{eq:fourth} and the fact that $\mu\in C(\mathbb{R};\mathcal{M}_+(\mathbb{R}))$; for the methodology, see the argument in \cite[Remark 1.1]{gao2021regularity} for instance. Next, we are going to show the conservation of $\tilde{F}(u(t),\mu(t))$, which will play an essential role for proving the orbital stability of peakons. It is worth noting that for a classical solution $u$, the following conservation law holds:
\begin{align}
[u(u^2+u_x^2)]_t+[u^2(u^2+u_x^2)]_x+\left(Pu^2-\frac{3}{4}u^4+P^2-P_x^2\right)_x=0.
\end{align}
This motivates us to consider the following lemma for weak solutions:
\begin{lemma}\label{lmm:weakF}
Let $(u,\mu)$ be a conservative solution to the generalized framework \eqref{eq:gCH1}-\eqref{eq:gCH3} subject to the initial datum $(\bar{u},\bar{\mu})\in\mathcal{D}$ in the sense of Definition~\ref{def:Conservative solutions}. Then for any $\phi\in C_c^\infty(\mathbb{R}\times\mathbb{R})$, we have
\begin{align}\label{eq:weakF}
\int_{\mathbb{R}}\int_{\mathbb{R}}(\phi_t+u\phi_x)u\di\mu(t)\di t+\int_{\mathbb{R}}\int_{\mathbb{R}} \phi_x\left(Pu^2-\frac{3}{4}u^4+P^2-P_x^2\right) \di x\di t=0,
\end{align}
where $P=\frac{1}{4}\varphi\ast \mu+\frac{1}{4}\varphi\ast u^2$.
\end{lemma}
\begin{proof}
Using \eqref{eq:mut} and \eqref{eq:g1}, we can apply a change of variables, and obtain
\begin{equation}\label{eq:twoterms}
\begin{aligned}
\int_{\mathbb{R}}\int_{\mathbb{R}}(\phi_t+u\phi_x)u\di\mu(t)\di t &=\int_{\mathbb{R}}\int_{\mathbb{R}}\frac{\di}{\di t}\phi(y(\alpha,t),t)u(y(\alpha,t),t)(\beta_\alpha-y_\alpha)(\alpha,t)\di\alpha\di t\\
&=-\int_{\mathbb{R}}\int_{\mathbb{R}}\phi(y(\alpha,t),t)\frac{\di}{\di t}[u(y(\alpha,t),t)(\beta_\alpha-y_\alpha)(\alpha,t)]\di\alpha\di t ,
\end{aligned}
\end{equation}
where we integrated by parts in the last equality. It follows from \cite[Lemma 2.4]{holden2007global} that for a.e. $(\alpha,t)\in \mathbb{R}^2$,
\begin{equation}\label{eq:int_diff}
\begin{aligned}
\frac{\partial }{\partial t} (\beta_\alpha-y_\alpha)(\alpha,t)= \frac{\partial }{\partial \alpha}(u^3-2uP) (y(\alpha,t),t).
\end{aligned}
\end{equation}
Combining \eqref{eq:g3} and \eqref{eq:int_diff} yields that for a.e. $(\alpha,t)\in \mathbb{R}^2$,
\begin{multline*}
 \frac{\partial }{\partial t} \left[u(y(\alpha,t),t)(\beta_\alpha-y_\alpha)(\alpha,t) \right] = -P_x(y(\alpha,t),t)(\beta_\alpha -y_\alpha)(\alpha,t)+ u(y(\alpha,t),t)\frac{\partial }{\partial \alpha}(u^3-2uP) (y(\alpha,t),t)\\
 = -P_x(y(\alpha,t),t)(\beta_\alpha -y_\alpha)(\alpha,t)+ \frac{3 }{4}\frac{\partial }{\partial \alpha}\left[u^4 (y(\alpha,t),t)\right]-2u(y(\alpha,t),t) \frac{\partial }{\partial \alpha}\left[(uP)(y(\alpha,t),t)\right].
\end{multline*}
Using the above identity, we can re-write \eqref{eq:twoterms} as follows:
\begin{equation}\label{eq:IntegralLHS}
\begin{aligned}
&\quad \int_{\mathbb{R}}\int_{\mathbb{R}}(\phi_t+u\phi_x)u\di\mu(t)\di t \\
&=\int_{\mathbb{R}}\int_{\mathbb{R}} \phi(y(\alpha,t),t)P_x(y(\alpha,t),t)(\beta_\alpha-y_\alpha)(\alpha,t)\di\alpha\di t\\
&\qquad -\int_{\mathbb{R}}\int_{\mathbb{R}}\phi(y(\alpha,t),t)\left(\frac{3 }{4}\frac{\partial }{\partial \alpha}\left[u^4 (y(\alpha,t),t)\right]-2u(y(\alpha,t),t) \frac{\partial }{\partial \alpha}\left[(uP)(y(\alpha,t),t)\right]\right)\di\alpha\di t\\
&=\int_{\mathbb{R}}\int_{\mathbb{R}} \phi(y(\alpha,t),t)P_x(y(\alpha,t),t)(\beta_\alpha-y_\alpha)(\alpha,t)\di\alpha\di t\\
&\qquad+\int_{\mathbb{R}}\int_{\mathbb{R}}\phi_x\left(\frac{3}{4}u^4-Pu^2\right)\di x\di t+\int_{\mathbb{R}}\int_{\mathbb{R}}\phi u^2P_x\di x\di t.
\end{aligned}
\end{equation}
Comparing \eqref{eq:IntegralLHS} with \eqref{eq:weakF}, we only need to show
\begin{equation}\label{eq:last}
\begin{aligned}
&\quad -\int_{\mathbb{R}}\int_{\mathbb{R}}\phi_x(P^2-P_x^2)\di x\di t \\
&=\int_{\mathbb{R}}\int_{\mathbb{R}} \phi(y(\alpha,t),t)P_x(y(\alpha,t),t)(\beta_\alpha-y_\alpha)(\alpha,t)\di\alpha\di t+\int_{\mathbb{R}}\int_{\mathbb{R}}\phi u^2P_x\di x\di t.
\end{aligned}
\end{equation}
It follows from part (ii) of Definition~\ref{def:Conservative solutions} that $\di\mu(t)=(u^2+u_x^2)(\cdot,t)\di x$ for a.e. $t$. At these times $t$, we have
\[
P(x,t)=\frac{1}{4}\int_{\mathbb{R}}\varphi(x-y)\di\mu(t)+\frac{1}{4}\int_{\mathbb{R}}\varphi(x-y)u^2(y,t)\di y=\frac{1}{2}\int_{\mathbb{R}}\varphi(x-y)\left(u^2+\frac{u_x^2}{2}\right)(x,t)\di x,
\]
which implies
\[
(P-P_{xx})(x,t)=\left(u^2+\frac{u_x^2}{2}\right)(x,t).
\]
Hence,
\begin{equation}
\begin{aligned}
&\quad -\int_{\mathbb{R}}\int_{\mathbb{R}}\phi_x(P^2-P_x^2)\di x\di t=-\int_{A}\int_{\mathbb{R}}\phi_x(P^2-P_x^2)\di x\di t\\
&=\int_{A}\int_{\mathbb{R}}2\phi P_x(P-P_{xx})\di x\di t=\int_{A}\int_{\mathbb{R}}2\phi P_x\left(u^2+\frac{u_x^2}{2}\right)(x,t)\di x\di t\\
&=\int_{A}\int_{\mathbb{R}}\phi u^2P_x\di x\di t+\int_{A}\int_{\mathbb{R}}\phi P_x(u^2+u_x^2)\di x\di t\\
&=\int_{\mathbb{R}}\int_{\mathbb{R}}\phi u^2P_x\di x\di t+\int_{A}\int_{\mathbb{R}}\phi(y(\alpha,t),t)P_x(y(\alpha,t),t)(\beta_\alpha-y_\alpha)(\alpha,t)\di\alpha\di t\\
&=\int_{\mathbb{R}}\int_{\mathbb{R}}\phi u^2P_x\di x\di t+\int_{\mathbb{R}}\int_{\mathbb{R}}\phi(y(\alpha,t),t)P_x(y(\alpha,t),t)(\beta_\alpha-y_\alpha)(\alpha,t)\di\alpha\di t.
\end{aligned}
\end{equation}
This verifies \eqref{eq:last}, and hence, completes the proof of \eqref{eq:weakF}.
\end{proof}
Using Lemma \ref{lmm:weakF}, we have the following lemma for conservation of $\tilde{F}$:
\begin{lemma}\label{lmm:Fconservation}
Let $(u,\mu)$ be a conservative solution to the generalized framework \eqref{eq:gCH1}-\eqref{eq:gCH3} subject to the initial datum $(\bar{u},\bar{\mu})\in\mathcal{D}$ in the sense of Definition~\ref{def:Conservative solutions}. Then $\tilde{F}(u,\mu)$ is conserved, namely $\tilde{F}(u(t),\mu(t)) = \tilde{F}(\bar{u},\bar{\mu})$ for all $t\in\mathbb{R}$.
\end{lemma}
The proof of Lemma~\ref{lmm:Fconservation} basically follows from \eqref{eq:weakF}, and shares some similarity with \cite[Remark 1.1]{gao2021regularity}. For the sake of self-containedness, we will provide the details here.
\begin{proof}
Without loss of generality, let us consider any arbitrary time $t>0$; the case for $t<0$ will be basically similar, so we will leave this to interested readers.
We choose non-negative smooth  functions $\chi_\e(s)$ and $g_{R}(x)$ for $\e>0$ and $R>0$, where 	$\chi_\e(s)=0$ for $s \le -\e$ and $s \ge t+\e$. 
$\chi_\e(s)=1$ for $s\in (0, t)$,  $\chi_\e'(s) \ge 0$ for $s \in (-\e,0)$ and $\chi_\e'(s) \le 0$ for $s \in (t, t+\e)$.  The function $g_R$ satisfies $g_{R}(x) = 1$ for $|x| \le R$, $g_{R}(x) = 0$ for $|x| \ge 2R$ and $\left|g'_{R}(x)\right| \le \frac{2}{R}$. Finally, let  $\phi(x,t) = \chi_\e(t) g_{R}(x)$, and substituting this $\phi$ into equation \eqref{eq:weakF}, we obtain  
\begin{multline*}
\int_{\mathbb{R}} \int_{\mathbb{R}} \chi'_\e (s) g_{R}(x) u(x,s)\di \mu(s)\di s + \int_{\mathbb{R}} \int_{\mathbb{R}}  \chi_\e (s) g'_{R}(x) u^2(x, s)\di \mu(s)  \di s \\
	+\int_{\mathbb{R}}\int_{\mathbb{R}}\chi_\e (s) g'_{R}(x) \left(Pu^2-\frac{3}{4}u^4+P^2-P_x^2\right)(x,s) \di x\di s=0.
\end{multline*}
Passing to the limit as $R \to \infty$ yields
\begin{equation}\label{eq:afterRtoinfty}
\int_{-\e}^{0}\chi'_\e (s) \int_{\mathbb{R}}u(x,s)\di\mu(s)  \di s + \int_{t}^{t+\e}\chi'_\e (s)\int_{\mathbb{R}}u(x,s)\di\mu(s) \di s =0.
\end{equation}
It is worth noting that
\begin{equation*}
\begin{aligned}
\left|	\int_{-\e}^{0}\chi'_\e (s) \int_{\mathbb{R}}u(x,s)\di\mu(s)  \di s -  \tilde{F}(\bar{u},\bar{\mu})\right|
=&\left|	\int_{-\e}^{0}\chi'_\e (s) \int_{\mathbb{R}}u(x,s)\di\mu(s)  \di s - 	\int_{-\e}^{0}\chi'_\e (s) \int_{\mathbb{R}}\bar{u}(x)\di\bar{\mu}  \di s \right|\\
\le& \int_{-\e}^{0}\chi'_\e (s) \left|\int_{\mathbb{R}}u(x,s)\di\mu(s)  - \int_{\mathbb{R}}\bar{u}(x)\di\bar{\mu}  \right| \di s,
\end{aligned}
\end{equation*}
and
\[
\begin{aligned}
\left|\int_{\mathbb{R}}u(x,s)\di\mu(s)  - \int_{\mathbb{R}}\bar{u}(x)\di\bar{\mu}  \right| &\leq \left|\int_{\mathbb{R}}[u(x,s)-\bar{u}(x)]\di\mu(s)  \right|+\left|\int_{\mathbb{R}}\bar{u}(x)\di\mu(s) - \int_{\mathbb{R}}\bar{u}(x)\di\bar{\mu}  \right|\\
&=: I_1(s)+I_2(s).
\end{aligned}
\]
Since $\mu \in C(\mathbb{R};\mathcal{M}_+(\mathbb{R}))$, we have $I_2(s)\to0$ as $s\to0$. For the first term $I_1$, we use the fact that the solutions $u(\cdot,t)$ is continuous\footnote{See \cite[Theorem 4.2]{holden2007global} and \cite[Proposition 5.2]{holden2007global} for instance.} in time in the space $L^\infty(\mathbb{R})$, and conclude that $I_1(s)\to0$ as $s\to0$.
This shows that $\int_{-\e}^{0}\chi'_\e (s)\int_{\mathbb{R}}u(x,s)\di\mu(s) \di s \to  \tilde{F}(\bar{u},\bar{\mu})$ as $\e \to 0^+$.
On the other hand, one can also verify that $ \int_{t}^{t+\e}\chi'_\e (s)\int_{\mathbb{R}}u(x,s)\di\mu(s) \di s  \to - \tilde{F}(u(t),\mu(t))$ as $\e \to 0^+$ in a similar manner, and hence, passing to the limit $\epsilon\to 0^+$ in \eqref{eq:afterRtoinfty} yields $\tilde{F}(u(t),\mu(t))=  \tilde{F}(\bar{u},\bar{\mu})$. Since the time $t$ is arbitrarily chosen, this completes the proof.

\end{proof}
\begin{remark}[A formal proof]\label{rmk:formal}
Ignoring the regularity issue, one could formally prove the conservation of $\tilde{F}$ more directly as follows:
\begin{equation}\label{eq:issue}
\begin{aligned}
&\;\quad\frac{\di}{\di t}\tilde{F}(u,\mu)\\ &=\frac{\di}{\di t}\int_{\mathbb{R}}u(y(\alpha,t),t)[\beta_\alpha(\alpha,t)-y_\alpha(\alpha,t)]\di\alpha\\
&=-\int_{\mathbb{R}}P_x(y(\alpha,t),t)\left[\beta_\alpha(\alpha,t) -y_\alpha(\alpha,t)\right]\di\alpha +\int_{\mathbb{R}}[3u^3u_x -  2u(uP)_x](y(\alpha,t),t)y_{\alpha }(\alpha,t)\di\alpha\\
&= - \int_{\mathbb{R}}P_x(y(\alpha,t),t)\left[\beta_\alpha(\alpha,t) -y_\alpha(\alpha,t)\right]\di\alpha + \frac{3}{4}\int_{\mathbb{R}} \partial_\alpha u^4( y(\alpha,t),t) \di \alpha\\
&\quad +\int_{\mathbb{R}} \partial_\alpha [u^2(y(\alpha,t),t)]P(y(\alpha,t),t) \di\alpha\\
&= - \int_{\mathbb{R}}P_x(y(\alpha,t),t)\left[\beta_\alpha(\alpha,t) -y_\alpha(\alpha,t)\right]\di\alpha -
\int_{\mathbb{R}}u^2(y(\alpha,t),t)P_x(y(\alpha,t),t)y_\alpha(\alpha,t)\di\alpha
\end{aligned}
\end{equation}
Now using \eqref{eq:Pderivative}, we have
\begin{equation}
\begin{aligned}
\frac{\di}{\di t}\tilde{F}(u,\mu)	=&\int_{\mathbb{R}} \int_{\mathbb{R}}\sgn(\alpha-\theta)e^{-|y(\alpha,t)-y(\theta,t)|}[\beta_\theta(\theta,t)-y_\theta(\theta,t)+u^2(y(\theta,t),t)y_\theta(\theta,t)]\\
&\qquad\times [\beta_\alpha(\alpha,t)-y_\alpha(\alpha,t)+u^2(y(\alpha,t),t)y_\alpha(\alpha,t)]\di\alpha\di\theta=0.
\end{aligned}
\end{equation}
However, the regularity of weak solutions seems  not enough for taking the time derivative under the integral sign in the second line of \eqref{eq:issue}. This is the main reason why we had to use the ``integral form'' of this computation in the rigorous proof of Lemma~\ref{lmm:Fconservation}.
\end{remark}

\section{Stability of peakons}\label{sec:stable}

We will follow the ideas in \cite{constantin2000stability} to show the stability of one peakon under the generalized framework introduced in Section \ref{sec:general}; in particular, we will have to control the singular part of the energy measure. 
More precisely, we have the following
\begin{lemma}\label{lmm:onepeakonlem}
Let $c$ be a positive constant and $(u,\mu)\in\mathcal{D}$. Then
\begin{enumerate}[(i)]
\item for any $\xi\in\mathbb{R}$,
$$\tilde{E}(u, \mu)-E(c\varphi)=\mu_s(\mathbb{R})+\|u-c\varphi(\cdot-\xi)\|_{H^1}^2+4c(u(\xi)-c);$$
\item let $M:=\max_{x\in\mathbb{R}}u(x)$, then
\begin{align}\label{eq:EFrelation}
\tilde{F}(u,\mu)\leq M\tilde{E}(u, \mu)-\frac{2}{3}M^3.
\end{align}
\end{enumerate}
In addition, if there exists a constant $0<\delta<\min\{c/30,\sqrt{2}\}$, such that $\mu_s(\mathbb{R})+\|u-c\varphi\|_{H^1}^2<\delta^2/2$, then
\begin{enumerate}[(i)]
\setcounter{enumi}{2}
\item
\begin{align}\label{eq:distance}
|\tilde{E}(u, \mu)-E(c\varphi)|\leq 4c\delta \quad\mbox{and}\quad |\tilde{F}(u,\mu)-F(c\varphi)|\leq 8c^2\delta;
\end{align}
\item
$$|M-c|\leq\sqrt{6c\delta}.$$
\end{enumerate}
\end{lemma}
Compared with Lemmas 1-4 in \cite{constantin2000stability}, the key point is that our energy $\tilde{E}(u, \mu)$ and $\tilde{F}(u,\mu)$ are defined for all time $t$ and are both conserved, hence, the estimates can be applied for all time $t$. Moreover, our estimates contain the singular part $\mu_s(\mathbb{R})$, and this improves the estimates in Lemmas 1-4 in \cite{constantin2000stability}. It is worth noting that in \eqref{eq:distance} we use the coefficients $4$ and $8$ for convenience, although they are not optimal according to the  proof below.  One could also improve the coefficient $\frac{1}{30}$ in $0<\delta<\min\{c/30,\sqrt{2}\}$, but we will not pursue this here. These technical details will be left for interested readers.
\begin{proof}
\begin{enumerate}[(i)]
\item It follows from the definition of $\tilde{E}$ that
\[
\tilde{E}(u, \mu)-E(c\varphi)=\mu_s(\mathbb{R})+E(u)-E(c\varphi).
\]
A direct differentiation yields that $\varphi_x(x)=\varphi(x)$ for $x<0$ and $\varphi_x(x)=-\varphi(x)$ for $x>0$,  so we have
\begin{equation}
\begin{aligned}
\|u-c\varphi(\cdot-\xi)\|_{H^1}^2&=E(u)+E(c\varphi(\cdot-\xi))-2c\int_{\mathbb{R}}u(x)\varphi(x-\xi)\di x-2c\int_{\mathbb{R}}u_x(x)\varphi_x(x-\xi)\di x\\
&=E(u)+E(c\varphi(\cdot-\xi))-2c\int_{\mathbb{R}}u(x)\varphi(x-\xi)\di x\\
&\qquad-2c\int_{-\infty}^\xi u_x(x)\varphi(x-\xi)\di x+2c\int_{\xi}^{\infty} u_x(x)\varphi(x-\xi)\di x\\
&=E(u)-E(c\varphi)-4cu(\xi)+4c^2.
\end{aligned}
\end{equation}
Combining the above two identities, we obtain the desired result.

\item It follows from \cite[Lemma 2]{constantin2000stability} that
for any $u\in H^{1}(\mathbb{R})$, the quantities $E(u)$ and $F(u)$ defined by \eqref{eq:constant} satisfy
\[
F(u)\leq ME(u)-\frac{2}{3}M^3, ~M=\max_{x\in\mathbb{R}}u(x),
\]
then for any $(u,\mu)\in \mathcal{D}$, we have
\begin{equation*}
\tilde{F}(u,\mu)=\int_{\mathbb{R}}u\di\mu_s+F(u)\leq M\mu_s(\mathbb{R})+ME(u)-\frac{2}{3}M^3=M\tilde{E}(u, \mu)-\frac{2}{3}M^3.
\end{equation*}

\item Due to the hypothesis $\mu_s(\mathbb{R})+\|u-c\varphi\|_{H^1}^2<\delta^2/2<1$, we also have
\begin{equation}\label{eq:delta_est}
\begin{aligned}
\mu_s^{1/2}(\mathbb{R})+\|u-c\varphi\|_{H^1}<\delta,
\end{aligned}
\end{equation}
and by Morrey's inequality,
\begin{align}
\|u\|_{L^\infty}\leq \frac{1}{\sqrt{2}}\|u\|_{H^1} \leq \frac{1}{\sqrt{2}}(\|c\varphi\|_{H^1}+\delta)=c+\frac{\delta}{\sqrt{2}}.
\end{align}
A direct computation yields $\|c\varphi\|_{H^1}=\sqrt{2}c$. Since $\delta<c/2$ and \eqref{eq:delta_est}, we have
\begin{equation*}
\begin{aligned}
|\tilde{E}(u, \mu)-E(c\varphi)|&=|\mu_s(\mathbb{R})+E(u)-E(c\varphi)|
\leq\mu_s(\mathbb{R})+|(\|u\|_{H^1}-\|c\varphi\|_{H^1})(\|u\|_{H^1}+\|c\varphi\|_{H^1})|\\
&\leq\mu_s(\mathbb{R})+(2\sqrt{2}c+\delta)\|u-c\varphi\|_{H^1}
\leq\frac{\delta^2}{2}+\delta(2\sqrt{2}c+\delta)\leq 4c\delta,
\end{aligned}
\end{equation*}
and by H\"older's inequality
\begin{equation*}
\begin{aligned}
&\quad\; |\tilde{F}(u,\mu)-F(c\varphi)| \leq \left|\int_{\mathbb{R}}u\di\mu_s\right|+|F(u)-F(c\varphi)|\\
&\leq\left(c+\frac{\delta}{\sqrt{2}}\right)\mu_s(\mathbb{R})+\left|\int_{\mathbb{R}}(u-c\varphi)(u^2+u_x^2)\di x+\int_{\mathbb{R}}c\varphi(u^2+u_x^2-c^2\varphi^2-c^2\varphi_x^2)\di x\right|\\
&\leq\left(c+\frac{\delta}{\sqrt{2}}\right)\frac{\delta^2}{2}+\left|\int_{\mathbb{R}}(u-c\varphi)(u^2+u_x^2)\di x+\int_{\mathbb{R}}c\varphi[(u-c\varphi)^2+(u_x-c\varphi_x)^2]\di x\right.\\
&\qquad\left.+\int_{\mathbb{R}}c\varphi[2(u-c\varphi)c\varphi+2(u_x-c\varphi_x)c\varphi_x]\di x\right|\\
&\leq\left(c+\frac{\delta}{\sqrt{2}}\right)\frac{\delta^2}{2}+\|u-c\varphi\|_{L^\infty}E(u)+\|c\varphi\|_{L^\infty}\|u-c\varphi\|_{H^1}^2+2\|c\varphi\|_{L^\infty}\|u-c\varphi\|_{H^1}\|c\varphi\|_{H^1}.
\end{aligned}
\end{equation*}
It follows from \eqref{eq:delta_est} that $\|u-c\varphi\|_{L^\infty}\leq \frac{1}{\sqrt{2}}\|u-c\varphi\|_{H^1}\leq \frac{\delta}{\sqrt{2}}$ and $E(u)= \|u\|^2_{H^1}\leq (\sqrt{2}c+\delta)^2$. Furthermore, a direct computation yields $\|c\varphi\|_{L^\infty}=c$ and $\|c\varphi\|_{H^1}=\sqrt{2}c$, so we finally have 
\begin{align*}
|\tilde{F}(u,\mu)-F(c\varphi)|\leq\left(c+\frac{\delta}{\sqrt{2}}\right)\frac{\delta^2}{2}+\frac{\delta}{\sqrt{2}}(2c^2+2\sqrt{2}c\delta+\delta^2)+c\frac{\delta^2}{2}+2c\delta\sqrt{2}c\leq 8c^2\delta.
\end{align*}

\item A direct calculation yields that
\begin{align}\label{eq:valueEF}
E(c\varphi)=2c^2\quad\mbox{and}\quad F(c\varphi)=\frac{4}{3}c^3.
\end{align}
Since \eqref{eq:valueEF} and $\delta<c/30$,  using \eqref{eq:EFrelation} and \eqref{eq:distance}, we have
\[
M\geq \frac{\tilde{F}(u,\mu)}{\tilde{E}(u, \mu)}\geq\frac{c}{2}.
\]
Define
\[
Q(y):=y^3-\frac{3}{2}y\tilde{E}(u, \mu)+\frac{3}{2}\tilde{F}(u,\mu),
\]
and
\[
Q_0(y):=y^3-\frac{3}{2}yE(c\varphi)+\frac{3}{2}F(c\varphi).
\]
Then it follows from part (ii) that
\begin{equation}\label{eq:Q(M)leq0}
Q(M)\leq 0.
\end{equation}
Furthermore, using \eqref{eq:valueEF}, we also have
\[
Q_0(y)=(y-c)^2(y+2c).
\]
Hence, using \eqref{eq:Q(M)leq0} and part (iii), we finally have
\begin{equation*}
\begin{aligned}
(M-c)^2(M+2c)&=Q_0(M)=Q(M)+\frac{3}{2}M(\tilde{E}(u, \mu)-E(c\varphi))-\frac{3}{2}(\tilde{F}(u,\mu)-F(c\varphi))\\
&\leq\frac{3}{2}M\cdot 4c\delta+\frac{3}{2}\cdot 8c^2\delta=6c(M+2c)\delta,
\end{aligned}
\end{equation*}
which implies $|M-c|\leq \sqrt{6c\delta}$.
\end{enumerate}
\end{proof}

Using Lemma~\ref{lmm:onepeakonlem}, we can immediately show our main result, namely Theorem~\ref{thm:onepeakonthm}, for the stability of peakons as follows.

\begin{proof}[Proof of Theorem \ref{thm:onepeakonthm}]
Let $\delta:=\frac{c\e^4}{6^4(1+c)^4}$, and $M_t:=\max_{x\in\mathbb{R}}u(t,x)$ for any $t\in\mathbb{R}$. Using Lemma~\ref{lmm:onepeakonlem} and the fact that $u(\xi(t),t)=M_t$, we have, for any $t\in\mathbb{R}$,
\begin{align*}
\mu_s(t)(\mathbb{R})+\|u(\cdot,t)-c\varphi(\cdot-\xi(t))\|_{H^1}^2\leq& |\tilde{E}(u(t), \mu(t))-E(c\varphi)|+4c|M_t-c|\\
=&|\tilde{E}(\bar{u}, \bar{\mu})-E(c\varphi)|+4c|M_t-c|\\
\leq&4c\frac{c\e^4}{6^4(1+c)^4}+4c\sqrt{6c\frac{c\e^4}{6^4(1+c)^4}}\leq\e^2.
\end{align*}

\end{proof}

\begin{remark}
Let $(u, \mu)$ be a conservative solution to  the generalized framework \eqref{eq:gCH1}-\eqref{eq:gCH3}  of the Camassa-Holm equation subject to the initial data $(\bar{u}, \bar{u}_x^2 \di x )$, which means that the initial energy measure $\bar{\mu}$ does not contain any singular part. If
\[
\|\bar{u}-c\varphi\|_{H^1}^2\leq\frac{c^2}{2}\left[\frac{\e}{6(1+c)}\right]^8,
\]
then for any $t\in\mathbb{R}$, we obviously  have
\[
\mu_s(t)(\mathbb{R}) + \|u(\cdot,t)-c\varphi(\cdot-\xi(t))\|_{H^1}^2\leq \e^2 ,
\]
because of Theorem \ref{thm:onepeakonthm},
where $\xi(t)\in\mathbb{R}$ is any point where the function $u(\cdot,t)$ attains its maximum. In particular, we know that for the usual $H^1(\mathbb{R})$ initial data, the stability of the peakon holds, even after wave breaking. This extends the  stability results of Constantin and Strauss \cite{constantin2000stability} globally-in-time.
\end{remark}

\noindent\textbf{Acknowledgements} Y. Gao is supported by  the National Natural Science Foundation grant 12101521 of China  and  the Start-up fund from The Hong Kong Polytechnic University with project number P0036186. 
T. K. Wong is partially supported by the HKU Seed Fund for Basic Research under the project code 201702159009, the Start-up Allowance for Croucher Award Recipients, and Hong Kong General Research Fund (GRF) grant ``Solving Generic Mean Field Type Problems: Interplay between Partial Differential Equations and Stochastic Analysis'' with project number 17306420, and Hong Kong GRF grant ``Controlling the Growth of Classical Solutions of a Class of Parabolic Differential Equations with Singular Coefficients: Resolutions for Some Lasting Problems from Economics'' with project number 17302521.

\bibliographystyle{plain}
\bibliography{bibofPeak}

\end{document}